\documentclass[11pt]{article}

\usepackage[utf8]{inputenc}
\usepackage{amsmath,amsthm}
\usepackage{amssymb,amsfonts}
\usepackage[sort]{cite}
\usepackage{enumitem}
\setenumerate{label=(\arabic*),itemsep=3pt,topsep=3pt}
\usepackage{subcaption}
\usepackage{hyperref}
\hypersetup{
   colorlinks = true,
   linkcolor = blue,
   anchorcolor = blue,
   citecolor = red,
   filecolor = blue,
   urlcolor = blue
}
\numberwithin{equation}{section}
\usepackage{tikz}
\usepackage[dvipsnames]{xcolor}
\usepackage[cmtip,all]{xy}
\newcommand{\longsquigarrow}{\xymatrix{{}\ar@{~>}[r]&{}}}
\usepackage{pmboxdraw}

\usepackage{listings}
\definecolor{maroon}{RGB}{133, 5, 63}
\definecolor{forestgreen}{RGB}{34, 139, 34}
\lstdefinelanguage{code}{
basicstyle=\small\ttfamily,
alsoletter=",
classoffset=1,
keywords={numerical_irreducible_decomposition, nid},
keywordstyle={\color{maroon}},
classoffset=2,
morekeywords={using, @var},
keywordstyle={\color{blue}},
classoffset=3,
morekeywords={julia, >},
keywordstyle={\color{forestgreen}},
xleftmargin=2cm,
xrightmargin=2cm,
columns=fullflexible,
keepspaces=true,
literate={•}{{\textbullet}}1
 {╭}{{\textSFi}}1 {╮}{{\textSFiii}}1 {╰}{{\textSFii}}1 {╯}{{\textSFiv}}1
 {─}{{\textSFx}}1 {│}{{\textSFxi}}1
 {├}{{\textSFviii}}1 {┤}{{\textSFix}}1
 {┬}{{\textSFvi}}1 {┴}{{\textSFvii}}1 {┼}{{\textSFv}}1,
}
\DeclareRobustCommand{\code}[1]{\lstinline!#1!}
\theoremstyle{plain} 
\newtheorem{theorem}{Theorem}[section]
\newtheorem{proposition}[theorem]{Proposition}
\newtheorem{corollary}[theorem]{Corollary}

\theoremstyle{definition} 
\newtheorem{definition}[theorem]{Definition}

\theoremstyle{remark} 
\newtheorem{remark}[theorem]{Remark}
\newtheorem{example}[theorem]{Example}

\usepackage[margin=1in]{geometry}
\begin{document}
\title{Numerical Irreducible Decomposition in \texttt{Julia}}

\author{Paul Breiding}
\date{}

\maketitle

\begin{abstract}
This article introduces a new implementation for computing a \emph{numerical irreducible decomposition} for a system of polynomial equations. The implementation is part of the software package \texttt{HomotopyContinuation.jl} \cite{homotopycontinuation} written in the programming language \texttt{Julia} \cite{julia}.

\end{abstract}
\bigskip

\bigskip
\section{Introduction}\label{sec:introduction}

Consider a system of polynomials with complex coefficients
$$F(x)\ =\ \begin{pmatrix}\; f_1(x)\;\; \\ \vdots\\\; f_k(x)\;\; \end{pmatrix},\qquad x = (x_1,\ldots, x_n).$$
Its zero set is an \emph{algebraic} variety 
 $$X = X(F)= \big\{x\in\mathbb C^n \mid f_1(x) = \cdots = f_k(x)=0\big\}.$$ 
\emph{Numerical irreducible decomposition} (NID) is an algorithmic framework for computing the \emph{irreducible components} 
\begin{equation}\label{irred_decomp}
X= X_1\cup\cdots\cup X_\ell
\end{equation}
of $X$ by numerical homotopy continuation. The goal is to compute a numerical representation of these components given $F$.

The software \texttt{HomotopyContinuation.jl}~\cite{homotopycontinuation} provides the function \code{nid} that computes a numerical irreducible decomposition. This article describes the implementation starting from v2.22.2.

\begin{example}\label{ex11}
Let us consider the following system of 3 polynomials in 3 variables from \cite{bates2024numericalnonlinearalgebra}:
 \begin{equation}\label{ex:nag}
  F(x,y,z)\ =\ \begin{pmatrix} p\cdot q \cdot (x-3)\cdot (x-5)\\ p\cdot q\cdot (y-3)\cdot (y-5)\\ p\cdot (z-3)\cdot (z-5) \end{pmatrix}, \quad\text{ where }
  \; \begin{array}{l}
p = p(x,y,z) = xy-x^2 - z + 1,\\[0.5em]
q = q(x,y,z) = x^4 + x^2 - y - 1.
\end{array}
 \end{equation}
The zero set $X(F)$ consists of a surface, two curves (one with $z=3$, one with $z=5$), and eight points. In total, there are 11 components. Figure \ref{fig1} shows the surface in blue, the two curves in red, and the eight points in green.

\begin{figure}
\begin{center}
\includegraphics[height = 7.5cm]{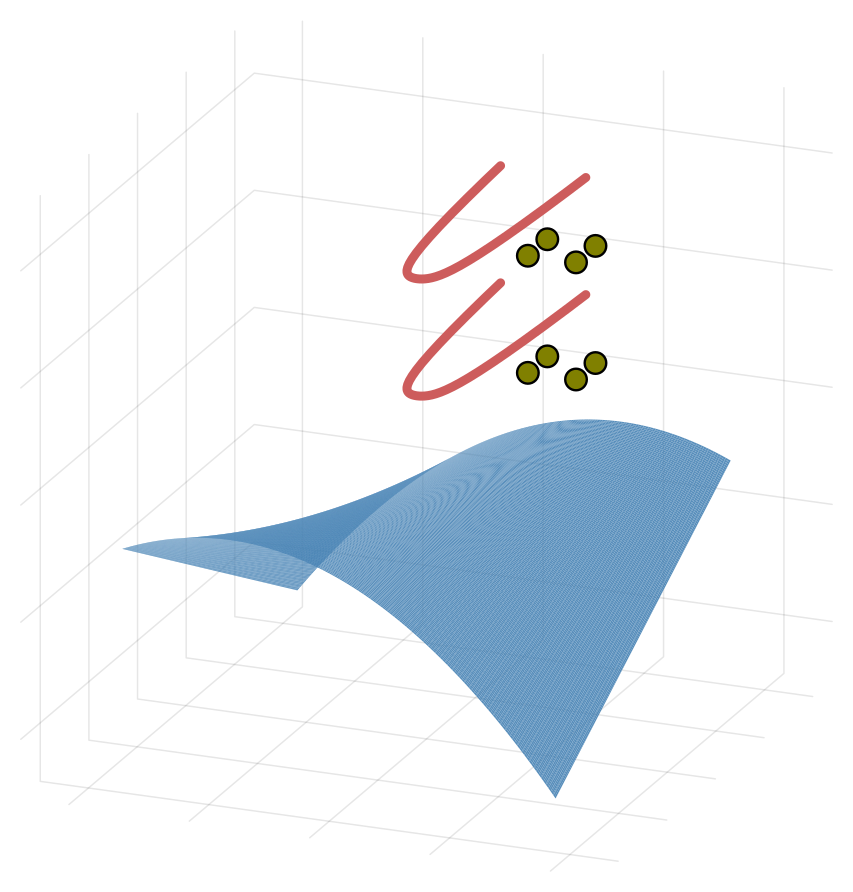}
\end{center}
\caption{\label{fig1}The zero set of the system of equations in \eqref{ex:nag}.}
\end{figure}

~

\vspace{-0.15cm}
We compute the decomposition of $X(F)$ in \texttt{HomotopyContinuation.jl}.
First, we set up the system $F$.
\begin{lstlisting}
julia> using HomotopyContinuation
julia> @var x y z
julia> p = x * y - x^2 - z + 1
julia> q = x^4 + x^2 - y - 1
julia> F = [p * q * (x - 3) * (x - 5);
            p * q * (y - 3) * (y - 5);
            p * (z - 3) * (z - 5)]
\end{lstlisting}

Then, we compute the numerical irreducible decomposition of $X(F)$ using the \code{nid} function.
\begin{lstlisting}
julia> N = nid(F)
Numerical irreducible decomposition with 11 components
======================================================
• 1 component(s) of dimension 2.
• 2 component(s) of dimension 1.
• 8 component(s) of dimension 0.
 degree table of components:
╭───────────┬──────────────────────────╮
│ dimension │  degrees of components   │
├───────────┼──────────────────────────┤
│     2     │            2             │
│     1     │          (4, 4)          │
│     0     │ (1, 1, 1, 1, 1, 1, 1, 1) │
╰───────────┴──────────────────────────╯
\end{lstlisting}

\noindent The output reports $11$ irreducible components, $1$ surface, $2$ curves, and 8 points. This is consistent with what we see in Figure \ref{fig1}.
The table at the end of the printed output shows the \emph{degrees} of the components (the degree is an algebraic invariant associated to each component; see Definition~\ref{def:degree}). The surface in Figure \ref{fig1} is a \emph{quadric surface} (it has degree $2$), and the curves are quartic curves (they both have degree $4$). A single point has, by definition, degree~$1$.

For the complete API of numerical irreducible decomposition in \texttt{HomotopyContinuation.jl} we refer to the \href{https://www.juliahomotopycontinuation.org/HomotopyContinuation.jl/stable/}{online documentation} available at \href{https://www.juliahomotopycontinuation.org}{juliahomotopycontinuation.org}.
\end{example}

The idea of numerical irreducible decomposition goes back to the works of Sommese, Verschelde and Wampler \cite{SVW2001, SVW2001b, SVW2002}. Implementations are available in \texttt{Bertini} \cite{BHSW06}, \texttt{Macaulay2} \cite{NumericalAlgebraicGeometryArticle}, and \texttt{PHCPack} \cite{SVW2003, PHCpack}. 
The contribution of this article is a new implementation in \texttt{HomotopyContinuation.jl}. In particular:
\begin{itemize}
\item We describe our implementation of numerical irreducible decomposition. It is based on the \emph{$u$-regeneration} algorithm by Duff, Leykin, and Rodriguez~\cite{u-regen} (see Sections \ref{sec:u_regen} and  \ref{implementation_u_regen}), and the monodromy decomposition algorithm (see Section \ref{sec:decomposition}).
\item We describe implementation choices that improve the numerical stability of $u$-regeneration, in particular the three-step formulation of the $u$-intersection step.
\item We prove Theorem \ref{thm_geodesics} describing geodesics in the complex Grassmannian and use this to formulate homotopies moving linear spaces.
\end{itemize}

\begin{example}
It is possible to call the two steps,  $u$-regeneration and decomposition,
individually. If \code{F} is a system of polynomials in \texttt{HomotopyContinuation.jl}, then 
\begin{lstlisting}
julia> regen = regeneration(F)
julia> dec = decompose(regen)
\end{lstlisting}
will first run $u$-regeneration and return the computed witness sets as \code{regen}. The second step decomposes the witness sets in \code{regen} and returns the decomposed witness sets as \code{dec}. To wrap \code{dec} into a \texttt{NumericalIrreducibleDecomposition} as in Example \ref{ex11} run 
\begin{lstlisting}
julia> NumericalIrreducibleDecomposition(dec)
\end{lstlisting}
\end{example}

The rest of the paper is organized as follows: 
The next section summarizes the theory behind witness sets, which is the main data structure used in NID. Section \ref{sec:u_regen} explains how $u$-regeneration works, and Section \ref{sec:linear_spaces} considers homotopies for linear slices. In Section \ref{sec:implementation} we discuss the details of the implementation and prove a correctness result (Theorem \ref{thm_correctness}).

\bigskip 
\section{Witness sets}\label{sec:ws}

In NID irreducible components are encoded by \emph{witness sets}. \code{N = nid(F)} computes a list of witness sets -- one for each irreducible component. \code{witness_sets(N)} returns a dictionary containing these witness sets.

In this section, we give a brief introduction to witness sets. For more details see \cite{BertiniBook, SommeseWamplerBook}. 

Recall that a variety $X$ is irreducible if every decomposition $X=X'\cup X''$ into varieties $X'$ and $X''$ implies that $X=X'$ or $X=X''$.
Witness sets encode irreducible varieties using linear intersections. The key property used is that irreducible varieties have a \emph{degree}. To define this, we introduce the \emph{affine Grassmannian}:
\begin{equation}\label{def:G_aff}
G_{\mathrm{aff}}(j,n) = \{L \subset \mathbb C^n \mid L \text{ is affine linear of dimension } j\}.
\end{equation}

\begin{definition}[Degree]\label{def:degree}
Let $X\subset \mathbb C^n$ be an irreducible variety of codimension $j$. There is a number $0<d<\infty$ such that for a generic affine linear subspace $L\in G_{\mathrm{aff}}(j,n)$ we have $$ \mathrm{deg}(X) := d = \#(X\cap L).$$ This number is called the \emph{degree} of $X$. 
\end{definition}
\begin{definition}[Witness sets, following the definition in \cite{breiding2026eliminationeliminatingcomputingcomplements}]\label{def:witness_set}~
\begin{enumerate}\item 
Let $X \subset \mathbb{C}^n$ be an irreducible variety of codimension $j$.
A \textit{witness set} for $X$ is a triple $(F,L, P)$, where 
\begin{enumerate}
\item[(a)] $F$ is a system of polynomials such that $X$ is an irreducible component of $X(F)$,
\item[(b)] $L\in G_{\mathrm{aff}}(j,n)$ is a generic affine linear subspace,
\item[(c)] $P = X \cap L$ contains $\mathrm{deg}(X)$ points.
\end{enumerate} 
$F$ is called a {\em witness system}, $L$ is the {\em witness slice}, and $P$ is the {\em witness point set}.
\item If $X$ is of pure codimension $j$, but not necessarily irreducible, let $X_1,\ldots,X_\ell$ be its irreducible components. We call $(F, L, P_1\cup \cdots\cup P_\ell)$ a witness set for $X$, where $P_i = X_i\cap L$.
\item If $W=(F, L, P)$ is a witness set, let $d = \# P$. We say that $W$ has degree $d$. 
\item We call the witness set $(F,L, P)$ \emph{reduced} if for all points $x\in P$ we have $\mathrm{rank}(JF(x)) = j$, where $JF(x)$ denotes the Jacobian of $F$ at $x$. 
\end{enumerate}
\end{definition}

\begin{remark}
The word \emph{generic} in both definitions means ``all points outside a certain algebraic subvariety''. This is slightly stronger than ``almost all points''.
\end{remark}

\begin{example}
As an example for a witness set, let us again consider the system from \eqref{ex:nag}. One irreducible component is the blue degree-2 surface in Figure \ref{fig1}. If we intersect this surface with a generic line, we get $2$ points. This can be seen in Figure \ref{fig2}. The system $F$ from \eqref{ex:nag} together with the red line and the two yellow points from Figure \ref{fig2} constitute a witness set for the blue surface. 

\begin{figure}
\begin{center}
\includegraphics[height = 7cm]{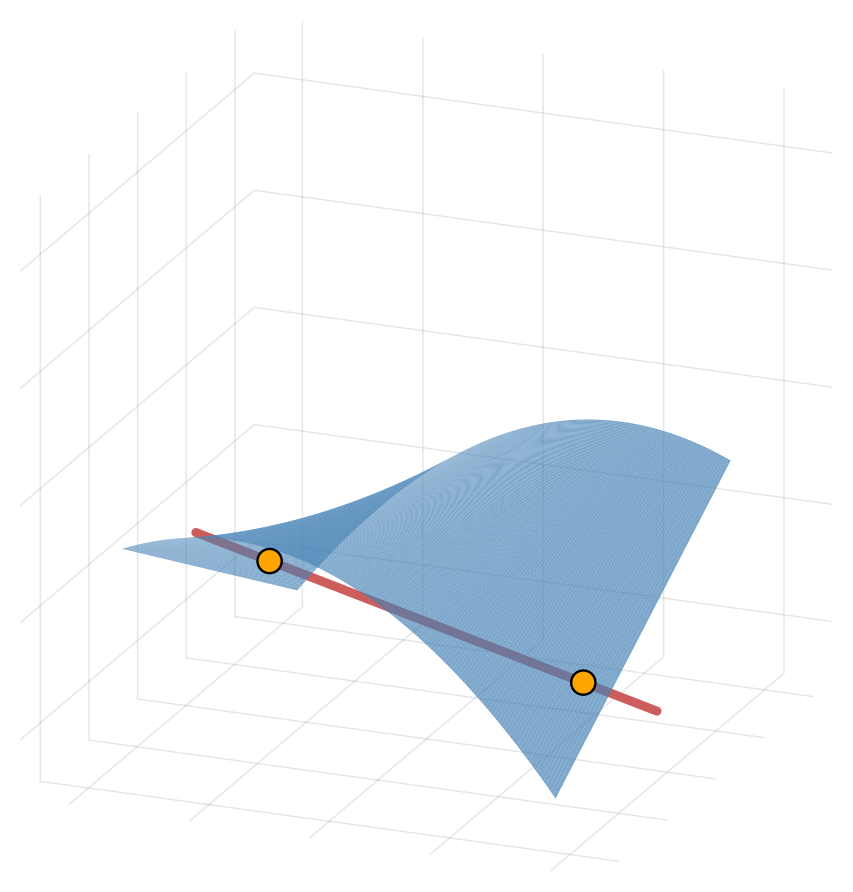}
\end{center}
\caption{\label{fig2} The degree-2 surface from Figure \ref{fig1} intersected with a line yields two points.}
\end{figure}
\end{example}

\begin{remark}\label{rem:NPD}
The implementation is restricted to reduced witness sets. Non-reduced components are discarded. If the Jacobian is rank-deficient at a point $x$, predictor-corrector methods in homotopy continuation fail at $x$. Therefore, neither $u$-regeneration nor monodromy decomposition can handle multiplicities, because they both take witness points as starting points for homotopy continuation. A numerical treatment of non-reduced structure would require numerical primary decomposition \cite{Leykin2008, CHEN20221, Manssour2021LinearPW}, which is beyond the scope of this article.
\end{remark}

Irreducibility of witness sets may not be known at construction time. When we only compute the witness points $P = X \cap L$, we do not know whether $P$ represents an irreducible variety or not. We have this information only after we have run the decomposition step (Section \ref{sec:decomposition}). 
In \texttt{HomotopyContinuation.jl} each witness set records whether it has been verified to be irreducible. 
\begin{example}[Irreducible witness sets] We compute a witness set for dimension $1$ directly. 
\begin{lstlisting}
julia> @var x y; f = x^2 + y^2 - 1
julia> W = witness_set(f; dim = 1);
julia> is_irreducible(W)
:undecided
\end{lstlisting} 
Here, \code{witness_set(f; dim = 1)} will directly solve $f(x,y)=\ell(x,y)=0$, where $\ell(x,y)$ is a random affine linear equations, and not run regeneration. 

Just computing the witness set does not tell us whether or not it represents an irreducible variety. This is why \code{is_irreducible(W)} returns \code{:undecided}. If we run \code{decompose} on \code{W}, then it will decompose it into irreducible components and assign to each the information that it is indeed irreducible:\enlargethispage{\baselineskip}
\begin{lstlisting}
julia> dec = decompose(W);
julia> W_irred = first(dec);
julia> is_irreducible(W_irred)
true
\end{lstlisting} 
Indeed, \code{is_irreducible(W_irred)} returns \code{true}, because \code{decompose} has assigned this information to~\code{W_irred}.
\end{example}

\subsection{What we can do with witness sets}
Witness sets can be used for sampling, for testing membership of points, and for computing intersections~\cite{bates2024numericalnonlinearalgebra,hauenstein2012whatisnag,IntersectingWitnessSets,SommeseWamplerBook}. \texttt{HomotopyContinuation.jl} provides functions for this. 

\begin{example}[Sampling]
We can sample points by moving the witness slice $L$ of a witness set $W=(F,L,P)$. The witness points $P$ move along, and we can collect them for sampling. Here is how to move linear spaces in \texttt{HomotopyContinuation.jl}.

We use the example from \eqref{ex:nag} and take a witness set for its codimension-1 component $\{p=0\}$, where $p =xy-x^2 - z + 1.$
\begin{lstlisting}
julia> @var x y z
julia> p = x * y - x^2 - z + 1
julia> W = witness_set(p; codim = 1)
Witness set for dimension 2 of degree 2
\end{lstlisting}

\noindent Consider the line $L=\{x=y=1\}$. We can use the constructor \code{LinearSubspace(A, b)} which defines the subspace $Ax=b$. 
\begin{lstlisting}
julia> A = [1 0 0; 0 1 0]; b = [1; 1]
julia> L = LinearSubspace(A, b)
1-dim. affine linear subspace {x | Ax=b}:
A:
[1.0 0.0 0.0; 0.0 1.0 0.0]
b:
[1.0, 1.0]
\end{lstlisting}
We move $W$ to $L$:
\begin{lstlisting}
julia> witness_set(W, L)
Witness set for dimension 2 of degree 1
\end{lstlisting}
The returned object has degree $1$, because for $x=y=1$ the equation $xy-x^2 - z + 1=0$ only has one zero and not two ($L$ is not a generic line). 
\end{example}

\begin{example}[Membership testing]\label{ex:membership}
We use the example from \eqref{ex:nag} and consider its codimension-2 part $\{q=(z-3)(z-5)=0\}$. 
\begin{lstlisting}
julia> @var x y z
julia> q = x^4 + x^2 - y - 1
julia> W = witness_set([q; (z - 3) * (z - 5)]; codim = 2)
Witness set for dimension 1 of degree 8
\end{lstlisting}
\noindent Recall that computing a witness set does not decide whether it represents an irreducible variety. In fact, \code{W} has degree 8 and represents the union of the two degree-4 curves in Figure \ref{fig1}. 

We can test membership in the variety represented by \code{W}. Let us first test a random point. 
\begin{lstlisting}
julia> pt = randn(3)
julia> membership(pt, W)
false 
\end{lstlisting}
\noindent Now, we check membership for the points in \code{W}. 
\begin{lstlisting}
julia> P = solutions(W)
julia> membership(P, W)
\end{lstlisting}
This returns a vector of \code{true} values of length $\mathrm{deg}(W)=8$. 
In particular, \code{membership} accepts both points and lists of points. 
We also test membership for the irreducible components. Let us first decompose and then use the first witness set for membership testing. 
\begin{lstlisting}
julia> dec = decompose(W)
julia> W_irred = first(dec)
julia> membership(P, W_irred)
\end{lstlisting}
This returns a boolean vector of length $8$, where 4 entries are true (those corresponding to \code{W_irred}).
\end{example}
~

\begin{example}[Intersecting a witness set with a hypersurface]
Consider a witness set $W$ and a hypersurface $\mathcal H$. We can compute witness sets for the intersection of the variety represented by $W$ with $\mathcal H$. To illustrate this, we take the witness set \code{W} from the previous example and intersect it with the hypersurface defined by $g =x^2 + y^2 + z^2 - 1$. 
\begin{lstlisting}
julia> g = x^2 + y^2 + z^2 - 1
julia> intersect(W, g)
Witness set for dimension 0 of degree 16
\end{lstlisting}
The variety represented by $W$ and the sphere $\mathcal H=\{g=0\}$ intersect transversely, so their intersection is zero-dimensional. By intersecting witness sets we get $16 =2\cdot 8 = \mathrm{deg}(g) \cdot \mathrm{deg}(W)$ points. 
\end{example}

\bigskip
\section{$u$-regeneration}\label{sec:u_regen}

The implementation of \code{nid} is based on an algorithm called $u$-regeneration, which was introduced in \cite{u-regen}. Regeneration is the name for a class of \emph{equation-by-equation} solvers. Instead of solving a system of polynomials directly, we successively intersect witness sets with hypersurfaces. Compared to directly solving the full system this can significantly reduce the number of paths that need to be tracked, especially when a system of equations has components in multiple dimensions.

Here is how $u$-regeneration works. 

$u$-regeneration uses polynomial homotopy continuation. 
When we have a homotopy from a system of polynomials $F_1$ to another system $F_2$ we write this as 
$$ F_1 \longsquigarrow F_2.$$
For a detailed introduction to polynomial homotopy continuation see the textbooks \cite{BertiniBook, SommeseWamplerBook}. 

Suppose that we have a system of polynomials $F=(f_1,\ldots, f_k)$. We wish to compute a witness set for each equidimensional part of the variety $X(F)$. These witness sets do not need to represent irreducible components. One witness set per codimension is enough. The decomposition algorithm is not part of $u$-regeneration. We discuss this separate step in Section~\ref{sec:decomposition}.

First, we sample a \emph{random flag} 
$$L_1\subset L_2\subset \cdots \subset L_{n-1}\subset L_n= \mathbb C^n,\qquad L_j\in G_\mathrm{aff}(j,n)$$
by sampling affine linear equations so that $L_j = \{\ell_1(x)=\cdots=\ell_{n-j}(x)=0\}$. 

Then, we compute a witness set for each hypersurface 
$$\mathcal H_i = X(f_i),\quad i=1,\ldots,k,$$
by slicing $\mathcal H_i$ with $L_1$; i.e., we directly solve $f_i(x) = \ell_1(x)=\cdots=\ell_{n-1}(x)=0$. 

Having computed the hypersurface witness sets, the algorithm proceeds iteratively taking $k$ steps. At the $i$-th step we have witness sets for equidimensional components of the intersection 
$$X(F_i) = \mathcal H_1\cap \cdots \cap \mathcal H_i,$$
where the $i$-th system is
$$F_i = \begin{pmatrix}f_1\\\vdots \\ f_i\end{pmatrix}.$$
The equidimensional components of $X(F_i)$ can have codimension at most $i$. At the $i$-th step we therefore have $\min(i,n)$ witness sets $W_1,\ldots,W_{\min(i,n)}$, where $W_j=(F_i, L_{j}, P_j)$ describes the codimension~$j$ part (possibly empty).

To proceed to the next step, we must compute witness sets for the equidimensional components of $X(F_{i+1}) = \mathcal H_1\cap \cdots \cap \mathcal H_i\cap \mathcal H_{i+1}$. This works as follows. First, initialize 
$$W_j' = (F_{i+1},\ L_{j},\ \emptyset),\quad j=1\ldots,\min(i+1, n),$$
where $W_j'$ stores the points in codimension $j$. 

For $j=1,\ldots,\min(i,n)$ we then process the witness set $W_j = (F_i, L_{j}, P_j)$ for codimension $j$ as follows: 
\begin{enumerate}
\item Check which of the points $P_j$ are on $\mathcal H_{i+1}$ by membership testing (see Example \ref{ex:membership}). Call these points $P_j'$. Add $P_j'$ to $W_j'$.\end{enumerate}
The following steps only need to be computed when $j<n$.
\begin{enumerate}[resume]
\item Let $Q_j:=P_j\setminus P_j'$ be the remaining points, and denote $d=\mathrm{deg}(f_{i+1})$. The points in the product $Q_j \times \left\{e^{2\pi\sqrt{-1}\, m/d} \mid m = 0,\ldots,d-1\right\}$ are used as start points for the next homotopy, called the \emph{$u$-intersection step}. Define a new variable $u$ and sample a random number $c\in\mathbb C$. 
The $u$-intersection step computes the homotopy 
\begin{equation}\label{u_step}
\begin{pmatrix}
f_1\\[0.2em]
\vdots \\[0.2em]
f_i\\[0.2em]
\textcolor{red}{u^d - 1}\\[0.2em]
\ell_1\\[0.2em]
\vdots\\[0.2em]
\ell_{n-(j+1)} \\[0.2em]
\textcolor{red}{\ell_{n-j}}
\end{pmatrix}
\longsquigarrow
\begin{pmatrix}
f_1\\[0.2em]
\vdots \\[0.2em]
f_i\\[0.2em]
\textcolor{red}{f_{i+1}}\\[0.2em]
\ell_1\\[0.2em]
\vdots\\[0.2em]
\ell_{n-(j+1)} \\[0.2em]
\textcolor{red}{u-c}
\end{pmatrix}
\end{equation}
(the red colors indicate those entries that change from left to right).
The endpoints of this homotopy lie on $(X(F_{i+1}) \cap L_{j+1})\times \{u=c\}$. 
\item 
Let $Q_j' \times \{c\}$ be the endpoints of the homotopy \eqref{u_step}. It is shown in \cite[Proposition 2.1]{u-regen} that~$Q_j'$ includes the witness points for $X(F_{i+1})$ in codimension $j+1$. We add them to $W_{j+1}'$. 
\end{enumerate}
After this loop we replace $W_j$ by $W_j'$ for $j=1,\ldots,\min(i+1, n)$.

At the end of this procedure we have computed witness sets 
$W_1, \ldots, W_n$
for $\mathcal H_1\cap \cdots \cap \mathcal H_k$, such that $W_j$ contains the witness points in codimension $j$. 

To obtain an irreducible decomposition we then decompose these witness sets into irreducible components (see Section \ref{sec:decomposition}).

\bigskip
\section{Moving linear spaces}\label{sec:linear_spaces}

In the $u$-intersection step, in membership testing \cite{SommeseWamplerBook} and in the monodromy decomposition algorithm we move linear spaces. The implementation in \texttt{HomotopyContinuation.jl} is based on geodesics in the Grassmannian. This section explains the theoretical background. 

Recall from \eqref{def:G_aff} the definition of the affine Grassmannian $G_{\mathrm{aff}}(j,n)$. In our implementation we are dealing with so-called \emph{slice systems}. 

\begin{definition}[Slice systems]
Consider a system of polynomials $F(x)$ and an affine linear space $L\in G_{\mathrm{aff}}(j,n)$. We define the \emph{slice system} $(F,L)$. 
\begin{enumerate}\item The slice system consists of two specific families of polynomial systems:
\begin{itemize}
\item the \emph{intrinsic systems} $F(Av + b)$, where $L=\{Av+b\mid v\in\mathbb C^j\}$;
\item the \emph{extrinsic systems} $(F(x),\ Wx-w)$, where $L=\{x\in\mathbb C^n \mid Wx-w=0\}$. 
\end{itemize}
Since the representation of $L$ as the image or the kernel of a matrix is not unique, we get a family of systems for each bullet point. 
\item 
When we say that we solve the slice system $(F,L)$ it means that depending on the context we pick one of the intrinsic systems or one of the extrinsic systems in $(F,L)$ and compute its zeros. 
\end{enumerate}
\end{definition}

Homotopy paths between slice systems in \texttt{HomotopyContinuation.jl} use geodesics in the \emph{linear Grassmannian}
$$G(j,n) = \{\mathcal L \subset \mathbb C^n \mid \mathcal L \text{ is linear of dimension } j\}.$$
We also denote the \emph{real linear Grassmannian}
$$G_{\mathbb R}(j,n)= \{\mathcal L \subset \mathbb R^n \mid \mathcal L \text{ is linear of dimension } j\},$$
and the complex and real Stiefel manifolds 
$$\mathrm{St}(j,n) := \{A \in\mathbb C^{n\times j} \mid A^*A = I\},\qquad \mathrm{St}_{\mathbb R}(j,n) := \{A \in\mathbb R^{n\times j} \mid A^TA = I\}.$$
One has
$$G(j,n) \cong \mathrm{St}(j,n) / U(j),\qquad G_{\mathbb R}(j,n) \cong \mathrm{St}_{\mathbb R}(j,n) / O(j)$$
as homogeneous spaces. Here, $U(j)$ and $O(j)$ denote the unitary and orthogonal group, respectively. This induces a metric on both $G(j,n)$ and $G_{\mathbb R}(j,n)$ using the standard metrics on $\mathbb C^{n\times j}$ and $\mathbb R^{n\times j}$. 
Lim, Sze-Wai Wong and Ye \cite{geodesics}, based on Edelman, Arias and Smith \cite[Theorem 2.3]{geodesics0}, describe geodesics in the real Grassmannian $G_{\mathbb R}(j,n)$ with respect to this metric. We use their results to prove the following. 

\begin{theorem}[Geodesics in the complex Grassmannian]\label{thm_geodesics}
Consider linear spaces
$\mathcal L_0 = \mathrm{image}(A_0)$ and $\mathcal L_1 = \mathrm{image}(A_1)$
defined by Stiefel matrices $A_0,A_1\in \mathrm{St}(j,n)$. A length-minimizing geodesic $\mathcal L: [0,1]\to  G(j,n)$ (with respect to the metric induced by the standard metric on $\mathrm{St}(j,n)$) from $\mathcal L_0$ to $\mathcal L_1$ is given by 
$$\mathcal L(t) =\mathrm{image}(A(t)), \qquad A(t) = A_0 U \cos(t\Theta) U^* + Q \sin(t\Theta)U^* \in \mathrm{St}(j,n),$$
where $U,V\in U(j)$, $Q\in\mathrm{St}(j,n)$, and $\Theta = \mathrm{diag}(\theta_1,\ldots, \theta_j)$ are given by the SVDs
$$A_0^*\, A_1 = U \cos(\Theta) V^*,\quad \text{and}\quad (I - A_0A_0^*)A_1 = Q\, \sin(\Theta)\, V^*.$$
\end{theorem}
\begin{remark}
\cite[Corollary 4.3]{geodesics} uses the SVD $(I - A_0A_0^*)A_1(A_0^*A_1)^{-1} = Q\, \tan(\Theta)\, U^*$ to compactly define all the data needed to establish the geodesic. However, computing $Q,U$ and $\Theta$ using this formulation is numerically ill-posed when a principal angle is (close to) $\frac{\pi}{2}$. This is the case when~$\mathcal L_0$ contains a vector that is (almost) perpendicular to $\mathcal L_1$. This motivates computing the SVDs of $A_0^*A_1$ and $(I - A_0A_0^*)A_1$ instead.
\end{remark}

We prove the theorem at the end of this section. Let us first discuss what it implies for moving slice systems.

Lim, Sze-Wai Wong and Ye \cite{geodesics} use the formula in Theorem \ref{thm_geodesics} to describe geodesics in the real affine Grassmannian. They embed an affine linear space $L\subset \mathbb R^n$ into projective space~$\mathbb P^{n}$ through the embedding $\iota: \mathbb R^n\hookrightarrow\mathbb P^n,\, \iota(x)= [x:1]$. Affine linear spaces are mapped to linear spaces under this embedding. They define the metric on the Grassmannian of real affine spaces $L\subset \mathbb R^n$ of dimension~$j$ to be the pullback metric via this map. The geodesic from an affine space~$L_0$ to another $L_1$ is then given as $L(t):=\iota^{-1}(\mathcal L(t))$, where $\mathcal L(t)\in G_{\mathbb R}(j+1,n+1)$ is the geodesic from $\mathcal L_0 = \iota(L_0)$ to~$\mathcal L_1 = \iota(L_1)$ from Theorem~\ref{thm_geodesics}. They prove that for almost all pairs of affine spaces, the path $\mathcal L(t)$ is indeed in the image of $\iota$. 

We found that the approach by Lim, Sze-Wai Wong and Ye \cite{geodesics} is numerically less well-conditioned compared to taking the approach that we discuss in the next subsection. This observation is supported by the plots in Figure \ref{cond_fig1}. We see that in the example of this figure the Intrinsic Subspace Homotopy has a condition number that is about two orders of magnitude smaller than the condition number of the homotopy that takes the path in the Grassmannian proposed by \cite{geodesics}.

\begin{figure}
\centering
  \begin{subfigure}{0.45\textwidth}
  \centering
  \includegraphics[height=5cm]{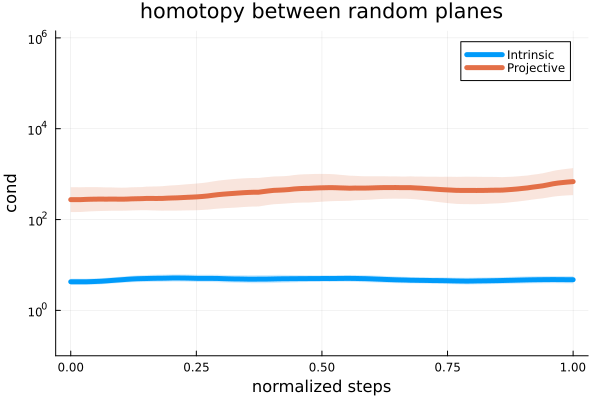}
  \caption{\label{F:hom1}}
  \end{subfigure}
\hfill
  \begin{subfigure}{0.45\textwidth}
  \centering
  \includegraphics[height=5cm]{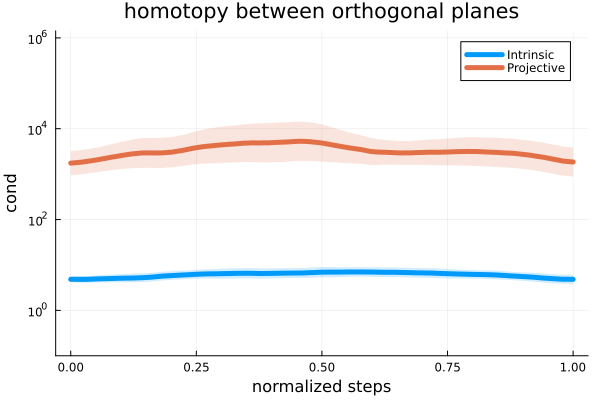}
  \caption{\label{F:hom2}}
  \end{subfigure}
\caption{\label{cond_fig1}The plots show the average condition numbers of the Jacobian for the Intrinsic Subspace Homotopy (blue) and the homotopy $F(A(t)v+b(t))$, where $A(t)v+b(t)$ follows the path proposed by Lim, Sze-Wai Wong and Ye~\cite{geodesics} (red, labelled ``projective''). Both homotopies move linear spaces inside a witness set $(F,L,P)$, where $F(x)=(q_1(x),\ q_2(x))$ consists of two quartic polynomials in $10$ variables so that $X(F)$ is irreducible of dimension $8$, $L$ is a plane in $\mathbb C^{10}$, and $W$ consists of $\mathrm{deg}(q_1)\cdot \mathrm{deg}(q_2) = 4\cdot 4= 16$ points. The~homotopies each start at the same witness points. The ticks on the $x$-axis correspond to the steps of the homotopies. The $y$-axis uses a logarithmic scale. The left plot shows the average condition numbers across 50 trials for moving between two random planes. The right plot shows the average condition numbers for moving between two random orthogonal planes. The shaded bands indicate 95\% confidence intervals across trials.}
\end{figure}

\subsection{Homotopies for slice systems}\label{sec_hom_slice}

Suppose that we have two slice systems $(F,L_0)$ and $(F, L_1)$ and we want to move $L_0$ into $L_1$ by a homotopy. Depending on the codimension of $L$ \texttt{HomotopyContinuation.jl} automatically chooses an \emph{Intrinsic Subspace Homotopy} when $\dim(L) \leq \mathrm{codim}(L)$, and an \emph{Extrinsic Subspace Homotopy} otherwise. We introduce these homotopies next.

Let us first discuss the intrinsic case. 
\begin{definition}[Intrinsic Subspace Homotopy]\label{def:ISH}
Suppose $(F,L_0)$ and $(F, L_1)$ are slice systems with
$$L_0 = \mathrm{image}(A_0) + b_0,\qquad L_1 = \mathrm{image}(A_1) + b_1,$$
where $A_0,A_1\in\mathrm{St}(j,n)$ are complex Stiefel matrices. The \emph{Intrinsic Subspace Homotopy} is
$$H(v,t) = F(A(t)v+b(t)),$$
where $b(t) = (1-t)b_0 + tb_1$ and $A(t)$ is as in Theorem \ref{thm_geodesics}.
\end{definition}
In Definition \ref{def:ISH} $A(t)$ is always a Stiefel matrix. 
The Jacobian in the Intrinsic Subspace Homotopy is 
$$\frac{\mathrm d }{\mathrm dv}\ H(v,t)= JF(A(t)v + b(t))\, A(t).$$ Choosing $A(t)$ as a Stiefel matrix avoids conditioning effects from a non-orthonormal representation, since then only the angles between the rows of $JF(A(t)v + b(t))$ and the columns of $A(t)$ contribute to conditioning, but not the internal representation of $\mathrm{image}(A(t))$.

Next, we discuss the extrinsic case. For this we first observe that the curve of matrices $A(t)$ in Theorem \ref{thm_geodesics} satisfies 
$A(0) =A_0$ and $A(1) = A_1VU^*,$ which has the same image as $A_1$.
\begin{definition}[Extrinsic Subspace Homotopy]\label{def:ESH}
Suppose that $(F,L_0)$ and $(F, L_1)$ are slice systems and that
$$L_0 = \{x\in\mathbb C^n \mid W_0x - w_0=0\},\qquad L_1 = \{x\in\mathbb C^n \mid W_1x - w_1=0\},$$
where $A_0:=W_0^*\in\mathrm{St}(n-j,n)$ and $A_1:=W_1^*\in\mathrm{St}(n-j,n)$ are complex Stiefel matrices. The \emph{Extrinsic Subspace Homotopy} is
$$H(x,t) = \begin{pmatrix} F(x)\\ A(t)^*x - w(t)\end{pmatrix},$$
where $w(t) = (1-t)w_0 + t\, (UV^*w_1)$ and $A(t),U,V$ are as in Theorem \ref{thm_geodesics}.
\end{definition}

The Jacobian in the Extrinsic Subspace Homotopy is $\begin{pmatrix} JF(x)\\ A(t)^*\end{pmatrix}$. As before we choose $A(t)$ to be a Stiefel matrix to avoid conditioning effects from a non-orthonormal representation.

\subsection{Proof of Theorem \ref{thm_geodesics}}
\begin{proof}
For any $s,r\geq 1$ consider the embedding $e:\mathbb C^{s\times r} \hookrightarrow \mathbb R^{(2s) \times (2r)}$ given by 
$$e(A + \sqrt{-1}\, B) = 
\begin{pmatrix} 
a_{1,1} & - b_{1,1} & \cdots & a_{1,r} & -b_{1,r} \\
b_{1,1} & a_{1,1} & & b_{1,r} & a_{1,r} \\
\vdots &&\ddots&&\vdots \\
a_{s,1} & -b_{s,1} & & a_{s,r} & -b_{s,r} \\
b_{s,1} & a_{s,1} & \cdots & b_{s,r} & a_{s,r}\end{pmatrix}, \quad A=(a_{i,j}),\ B=(b_{i,j})\in\mathbb R^{s\times r}$$
(we use the same symbol $e$ for every such embedding, although, strictly speaking we have an embedding $e=e_{s,r}$ for every $s,r$). 

We also denote $e(\mathcal L) := \mathrm{image}(e(A))$ for a complex subspace $\mathcal L=\mathrm{image}(A)\in G(j,n)$.
This sends $\mathcal L$ to the real subspace $e(\mathcal L)$ obtained by identifying $\mathbb C \cong \mathbb R^2$. One checks that this defines an isometric embedding 
$$e:G(j,n)\hookrightarrow G_{\mathbb R}(2j,2n).$$

For a matrix product $AB$ we have $e(AB) = e(A)e(B)$. Consequently, we have that $e(\mathcal L(t))$ is the image of 
$$e(A(t)) = e(A_0)\, e(U)\, e(\cos(t\Theta))\, e(U)^* + e(Q)\, e(\sin(t\Theta))\, e(U)^*,$$
where $e(Q), e(U), e(V)$ and $\Theta$ are obtained from the SVDs
\begin{align*} 
e(A_0)^*\, e(A_1) &= e(U)\,
\begin{pmatrix} 
\cos(\theta_1) & 0 & & & \\
0 & \cos(\theta_1) & & & \\
 &&\ddots&& \\
 & & & \cos(\theta_j) &0\\
 & & & 0& \cos(\theta_j)
\end{pmatrix}\,
e(V)^*,\\[0.9em]
\text{and}\quad 
(I - e(A_0)e(A_0)^*)\, e(A_1) &= e(Q)\, 
\begin{pmatrix} 
\sin(\theta_1) & 0 & & & \\
0 & \sin(\theta_1) & & & \\
 &&\ddots&& \\
 & & & \sin(\theta_j) &0\\
 & & & 0& \sin(\theta_j)
\end{pmatrix}
\, e(V)^*.
\end{align*} 
Indeed, the right-hand sides are SVDs, because $e(U)$ and $e(V)$ are orthogonal $(2j)\times (2j)$-matrices, and $e(Q)$ is a real Stiefel $(2n)\times (2j)$-matrix: $e(Q)^*e(Q) = e(Q^*Q) = e(I) = I$, and similar for $U$ and~$V$. By \cite[Corollary 4.3]{geodesics}, the path defined by $e(A(t))$ is a length-minimizing geodesic in
$G_{\mathbb R}(2j,2n)$ between its endpoints. Since $e$ preserves lengths, the corresponding
path $A(t)$ is length-minimizing in $G(j,n)$, hence is a geodesic.
\end{proof}
In particular, the proof implies the following, since every geodesic is locally length-minimizing.
\begin{corollary}
The embedding $e:G(j,n)\hookrightarrow G_{\mathbb R}(2j,2n)$ is totally geodesic. It maps length-minimizing geodesics in $G(j,n)$ to length-minimizing geodesics in $G_{\mathbb R}(2j,2n)$.
\end{corollary}

\bigskip
\section{Implementation details}\label{sec:implementation}

The implementation only considers reduced witness sets and discards points with higher multiplicity, since $u$-regeneration takes witness points as starting points for homotopies; see Remark~\ref{rem:NPD}. We have the following correctness result for our implementation. 

\begin{theorem}\label{thm_correctness}
For $1\leq i\leq k$, let $F_i=(f_1,\ldots,f_i)$. We say that a component $Y$ of $X(F_i)$ \emph{contributes} to a component~$Z$ of $X(F)$ if $Z$ is obtained from $Y$ by successively intersecting with the remaining hypersurfaces $\mathcal H_{i+1}, \ldots, \mathcal H_k$, where $\mathcal H_s = X(f_{s})$. 

Assume that for every $i$ all components of $X(F_i)$ which contribute to 
components of $X(F)$ that are reduced with respect to $F$ are themselves reduced with respect to $F_i$. Then $u$-regeneration computes witness sets $W_1,\ldots,W_{\min(k,n)}$ whose witness points are precisely the witness points of the reduced equidimensional parts of $X(F)$.
\end{theorem}
\begin{proof}
We prove the statement by induction on $k$. For $k=1$, we are only dealing with a single polynomial. The algorithm computes witness sets for the hypersurface $\mathcal H_1 = X(F_1)$ by solving
$f_1(x)=\ell_1(x)=\cdots=\ell_{n-1}(x)=0.$ We only keep the regular solutions of that system of equations. Therefore the solutions obtained in this first step are exactly the witness points on the components that are reduced with respect to $F_1$. We do not need the extra assumption for the induction start.

For the induction step we assume that $u$-regeneration has computed the witness
points of all reduced components of $X(F_{k-1})$ that contribute to reduced components of $X(F_k)$. We have to show that the same holds for $F=F_k$.

Let $Y$ be an irreducible component of $X(F_{k-1})$ of codimension $j$ that is reduced with respect to $F_{k-1}$, and let $P=Y\cap L_j$ be its witness point set. The membership test detects the points of $P$ lying on $\mathcal H_k$. These points are added to the codimension $j$ witness set. The remaining points are used as start points for the $u$-intersection step \eqref{u_step}. By \cite[Proposition 2.1]{u-regen}, the endpoints of this homotopy include the witness points of $Y\cap \mathcal H_k$ in codimension $j+1$. These points are added to the codimension $j+1$ witness set.\enlargethispage{\baselineskip}

Suppose that $Y\subseteq X(F_{k-1})$ contributes to $Z\subseteq X(F)$, and that $Z$ is reduced with respect to~$F$. By assumption, $Y$ is reduced with respect to $F_{k-1}$, so its witness points
are available by the induction hypothesis. Thus, we obtain $Z$ by $u$-regeneration in one of the two ways described above.
Conversely, if we obtain $Z$ by $u$-regeneration from $Y$, then $Y$ must have been reduced to begin with, and $Z$ is kept only if it is reduced with
respect to $F$.
\end{proof}

The assumption of Theorem \ref{thm_correctness} is not always satisfied. Consider the following example.

\begin{example}
Consider $F=(f_1,f_2)$ with $f_1(x,y) = (x-1)^2$ and $f_2(x,y) = x-1$. The zero set $X(F)$ is the line $x=1$, and it is reduced with respect to $F$.
The line $\mathcal H_1 = X(f_1)$ is not reduced with respect to~$f_1$, but it contributes to $X(F)$.
\end{example}

This issue with having intermediate non-reduced witness sets can be avoided by randomization. The following result is implicitly proved in \cite{HAUENSTEIN20111240}.
\begin{proposition}
Let $M\in\mathbb C^{k\times k}$ be a matrix with i.i.d.\ complex Gaussian entries, and $G:=M\cdot F$. With probability one, $G$ satisfies the assumption of Theorem \ref{thm_correctness}.
\end{proposition}

\texttt{HomotopyContinuation.jl} provides randomization as an option. While at first sight \emph{always} randomizing seems to make sense, it has two practical drawbacks. First, randomizing homogeneous systems may destroy the homogeneous structure. Second, randomization causes all degrees of $G$ to be equal to the maximal degree of the polynomials in $F$. Running $u$-regeneration on $G$ will therefore produce intermediate witness sets that can have significantly larger degree compared to the intermediate witness sets when running $u$-regeneration on $F$. Consequently, decomposing $X(G)$ can be much slower than decomposing $X(F)$. This is why the default mode in \texttt{HomotopyContinuation.jl} does not randomize (by default the degrees will be sorted in increasing order to keep the degrees of the intermediate witness sets small).

\begin{example}
We compute the example above.
\begin{lstlisting}
julia> @var x y
julia> f = System([(x - 1)^2; x - 1], variables = [x, y])
julia> regeneration(f; sorted = :randomized)
Witness set for dimension 1 of degree 1
\end{lstlisting}
The keyword \code{sorted} can be either \code{true} (default, sorts polynomials by degrees in increasing order), \code{false} (runs $u$-regeneration on the original $F$), or \code{:randomized} (randomizes $F$). \code{nid} also accepts the keyword \code{sorted}.
\end{example}

\subsection{$u$-regeneration}\label{implementation_u_regen}

The key part of $u$-regeneration is the $u$-intersection step \eqref{u_step}.
In the implementation we do not use a direct homotopy. Instead, we sample a random linear equation $a(x,u)=0$ and take an indirect path through $a$. Denote $d=\mathrm{deg}(f_{i+1})$ and 
$$F_i'(x,u):= \begin{pmatrix}
f_1\\[0.2em]
\vdots \\[0.2em]
f_i\\[0.2em]
u^d - 1
\end{pmatrix}, \qquad K_j:= L_j \cap \{a(x,u)=0\},\qquad L_j':= L_j \cap \{u-c=0\}.$$
Then, instead of taking the homotopy that goes directly $(F_i', L_j)\longsquigarrow (F_{i+1}, L_j')$ in \eqref{u_step}, we take the following three-step homotopy between slice systems:
$$(F_i', L_j) \longsquigarrow (F_i', K_j) \longsquigarrow (F_{i+1}, K_j) \longsquigarrow (F_{i+1}, L_j').$$
In particular, 
\begin{enumerate}
  \item $H_1: (F_i', L_j) \longsquigarrow (F_i', K_j)$ moves linear spaces
  \item $H_2: (F_i', K_j) \longsquigarrow (F_{i+1}, K_j)$ only moves $u^d-1$ to $f_{i+1}$ but keeps the rest fixed, and
  \item $H_3: (F_{i+1}, K_j) \longsquigarrow (F_{i+1}, L_j')$ again moves linear spaces. 
\end{enumerate}
This disentangles the (numerical) complexity of computing the $u$-intersection step. 

The first and the third homotopy move linear spaces, for which we use homotopies for slice systems (see Section \ref{sec_hom_slice}).

The crucial step is the second homotopy $H_2$. The start points of this homotopy are given by the Cartesian product of the witness points for $F_i$ in codimension $j$ with the roots of unity. These witness points can be far from the origin, which can make path tracking numerically challenging. To cope with this we have implemented a safety mechanism: run the first homotopy $H_1$ and check if all the endpoints of this homotopy can be used as start points for the second homotopy~$H_2$. If yes, continue. If no, resample $a$. We found that this step-by-step approach significantly increases the numerical stability of the computation.

\subsection{Decomposing witness sets}\label{sec:decomposition}

After $u$-regeneration we have a collection of witness sets $W_1,\ldots, W_n$, where $W_j = (F, L_j, P_j)$ represents the codimension $j$ part $Z_j$ of $X(F)$ that is reduced with respect to $F$. 

Suppose the decomposition of $Z_j$ into irreducible components is $$Z_j = X_j^{(1)}\cup \cdots \cup X_{j}^{(r_j)}.$$
To compute them we decompose $P_j$ as
$$P_j^{(i)} = X_j^{(i)}\cap L_j,\qquad i = 1,\ldots, r_j,$$
using the \emph{monodromy decomposition algorithm} as described in \cite{SommeseWamplerBook}: moving $L_j$ in a loop (a so-called \emph{monodromy loop}) induces a permutation on $P_j$. By \cite[Lemma A.12.1]{SommeseWamplerBook}, the orbits of this action under all such loops are precisely $P_j^{(i)}$ for $i=1,\ldots,r_j$. To understand when to stop we rely on the \emph{trace test} \cite{TraceTest}. This test takes a partial orbit and computes its \emph{trace}. Theoretically, the trace is zero if and only if the orbit is complete. Since we are dealing with numerical computation ``is zero'' is replaced by ``is smaller than some tolerance value'' (which by default is $10^{-6}$).
Thus, we proceed as follows.
\begin{enumerate}
\item Fix a codimension $j\in\{1,\ldots, n\}$.
\item Run monodromy loops for $L_j$ and collect the orbits within $P_j$. 
\item For every orbit run the \emph{trace test}. If the trace test succeeds, mark the orbit as completed. If not, run more monodromy loops.
\item End when all trace tests succeed.
\end{enumerate}
We obtain the witness sets $(F, L_j, P_j^{(i)})$ for $X_j^{(i)}$, $1\leq i\leq r_j$, for every codimension $j$. These represent the irreducible components of $X(F)$ that are reduced with respect to $F$.

\subsection{Projective varieties}

\texttt{HomotopyContinuation.jl} automatically detects if a system of polynomials is homogeneous. In this case, the system will automatically be sliced with linear spaces (instead of affine linear spaces) to keep the homogeneous structure. In addition, all reported dimensions are projective. Here is an example. 

\begin{lstlisting}
julia> @var x[1:4]
julia> a = x[1]^2 + x[2]^2 + x[3]^2 + x[4]^2
julia> b = x[1]^3 + x[2]^3 + 2x[3]^3 + 3x[4]^3
julia> c = x[1]^4 + 2x[2]^4 + 4x[3]^4 - x[4]^4
julia> g = [a * c; b * c]
\end{lstlisting}
The zero set of \code{g} consists of a surface in $\mathbb P^3$ of degree $4$ (defined by $c=0$), and of a curve in $\mathbb P^3$ of degree $6 = 2\cdot 3$ (defined by $a=b=0$):
\begin{lstlisting}
julia> nid(g) 
Numerical irreducible decomposition with 2 components
=====================================================
• 1 component(s) of dimension 2.
• 1 component(s) of dimension 1.

 degree table of components:
╭───────────┬───────────────────────╮
│ dimension │ degrees of components │
├───────────┼───────────────────────┤
│     2     │           4           │
│     1     │           6           │
╰───────────┴───────────────────────╯   
\end{lstlisting}

\subsection{Systems of rational functions}

Numerical irreducible decomposition also works for systems of rational functions. The implementation in 
\texttt{HomotopyContinuation.jl} allows such systems as input for \code{nid}.
\begin{lstlisting}
julia> @var x y z
julia> g = [x^2 + y^2 - z; x / (y - 1) + y + z - 1]
julia> nid(g)
Numerical irreducible decomposition with 1 component
====================================================
• 1 component(s) of dimension 1.

 degree table of components:
╭───────────┬───────────────────────╮
│ dimension │ degrees of components │
├───────────┼───────────────────────┤
│     1     │           4           │
╰───────────┴───────────────────────╯
\end{lstlisting}

However, based on experiments, computing the $u$-intersection step for more complicated rational systems appears to be numerically unstable. More research in regeneration for rational systems is required before this can be used safely. 

\subsection{Stress test}
The stress test for the implementation of numerical irreducible decomposition is the introductory example \eqref{ex:nag}. We ran $10,000$ sequential \code{nid} computations for this system and counted how often the output deviated from the true answer (one degree-2 surface, two degree-4 curves and 8 points). The final implementation produced the correct answer consistently throughout $10,000$ runs.

\bigskip 
\section*{Funding}
The author was supported by DFG, German Research Foundation -- Projektnummer 445466444.


\section*{Acknowledgments}

The implementation of numerical irreducible decomposition in \texttt{Julia} has enjoyed several iterations of improvements following extensive user feedback from silviana amethyst, Taylor Brysiewicz, Tobias Boege, Hannah Friedman, Jon Hauenstein, Oskar Henriksson, Max Hill, Ben Hollering, David K.\ Johnson, Frank Sottile, Bernd Sturmfels, Svala Sverrisd\'ottir, and Sascha Timme.  


\bigskip 
\section*{Usage of AI}

AI coding assistants have been used to support the implementation process. All implemented functions have been designed and reviewed by the author. 


\bigskip 
\bibliographystyle{abbrv}
\bibliography{literature}

\bigskip

{\samepage

\noindent {\bf Authors' address:}\\
Paul Breiding, University of Osnabr\"uck \hfill{\tt pbreiding@uni-osnabrueck.de}
}

\end{document}